\documentclass[preprint,12pt]{elsarticle}
\usepackage{amsfonts}

\usepackage{graphics}

\usepackage{amssymb}
\usepackage{amsthm}
\usepackage{amsfonts}
\usepackage{mathrsfs}
\usepackage{amscd,amssymb,amsmath,graphicx,verbatim,xy,bm,cases}
\usepackage[TS1,OT1,T1]{fontenc}
\newtheorem{theorem}{Theorem}[section]
\newtheorem{lemma}[theorem]{Lemma}

\theoremstyle{definition}
\newtheorem{definition}[theorem]{Definition}
\newtheorem{example}[theorem]{Example}

\newcommand{\E}{{\mathbb{E}}}
\newcommand{\R}{{\mathbb{R}}}
\theoremstyle{remark}

\catcode`@=11
\let\@int\int \def\int{\displaystyle\@int}
\let\@lim\lim \def\lim{\displaystyle\@lim}
\let\@sum\sum \def\sum{\displaystyle\@sum}
\let\@sup\sup \def\sup{\displaystyle\@sup}
\let\@inf\inf \def\inf{\displaystyle\@inf}
\let\@cap\cap \def\cap{\displaystyle\@cap}
\let\@cup\cup \def\cup{\displaystyle\@cup}
\let\@max\max \def\max{\displaystyle\@max}
\let\@min\min \def\min{\displaystyle\@min}
\let\@frac\frac \def\frac{\displaystyle\@frac}
\let\@iint\iint \def\iint{\displaystyle\@iint}

\def\epsilon{\varepsilon}

\numberwithin{equation}{section}

\journal{}

\begin{document}

\begin{frontmatter}
\title{Hitting time for one class of Gaussian processes}
\author[a]{Qingsong Wang}
\ead{qswang21@mails.jlu.edu.cn}
\author[b]{A. A. Dorogovtsev}
\ead{andrey.dorogovtsev@gmail.com}

\address[a]{School of Mathematics, Jilin University, 130012, Changchun, P. R. China}
\address[b]{Institute of Mathematics, National Academy of Sciences of Ukraine, Ukraine}

\begin{abstract}
In our paper, we consider Gaussian processes of the form
\begin{equation*}
\begin{aligned}
\eta(t)=\int_{0}^{1}A(\mathbf{1}_{[0;1]})(s)dw(s), \quad t\in[0;1],
\end{aligned}
\end{equation*}
where $w(t),t\geq0$ is a standard Wiener process in $\mathbb{R}^{d}$ and $A$ is continuous linear operator in $L_{2}([0;1])$. For the domain $\mathcal{D}\subset{\mathbb{R}}^{d}$ with smooth boundary $\partial\mathcal{D}$ the following probability is considered for $x\in\mathcal{D}$
\begin{equation*}
\begin{aligned}
P\{\exists\tau\in[0;1]:x+\eta(\tau)\in\partial \mathcal{D}\}.
\end{aligned}
\end{equation*}
It is well known that for Wiener process itself such hitting probabilities satisfy the parabolic boundary value problem for heat equation. Note that in general process $\eta$ is non-Markovian and even can be non-semi-martingale. Despite of this $\eta$ can be considered as an application of second quantization operator to $w$. Correspondingly we obtain a representation for hitting probabilities as a series of multiple integrals with the kernel obtained from Green function for domain $\mathcal{D}$ and tensor powers of operator $A$.
\end{abstract}
\begin{keyword}
anticipating problem; second quantization; Gaussian process.

\MSC{60G15, 60J45, 60J65.}
\end{keyword}
\end{frontmatter}

\section{Introduction}

In this paper we consider hitting times and hitting probabilities for the special class of Gaussian random processes which are called the integrators. These processes were introduced by A. A. Dorogovtsev in \cite{Dor06} for the purposes related to the theory of stochastic integration and anticipating stochastic differential equations. Integrators includes Brownian motion, Brownian bridge, Ornstein- Uhlenbeck process and many other frequently used in applications processes. It was discovered in that integrators can be used to model linear polymers. Such important characteristics as localtimes of self-inersection and non-normalized local times of self-intersection also were studied for integrator. Initially integrator was defined as follows.
\begin{definition}\cite{Dor06}
Centered Gaussian process $\eta(t),t\in[0;1]$ is called by integrator if there exists such positive $C>0$ that for arbitrary $n\geq1$, partition $0=t_{0}<\cdots<t_{n}=1$ and real numbers $a_{1},\cdots,a_{n}$
\begin{equation*}
\begin{aligned}
\E(\sum_{k=0}^{n-1}a_{k}(\eta(t_{k+1})-\eta(t_{k})))^{2}\leq C\sum_{k=0}^{n-1}a_{k}^{2}\Delta t_{k}^{2}.
\end{aligned}
\end{equation*}
\end{definition}
This condition means that for every function $f\in L_{2}([0;1])$ the integral
\begin{equation*}
\begin{aligned}
\int_{0}^{1}f(t)d\eta(t)
\end{aligned}
\end{equation*}
is well-defined as a limit in the square mean of the integrals from the step functions. Also, there is anther equivalent definition of the integrators through the Gaussian white noise in Hilbert space. This definition will be used later.

\begin{lemma}
Let $\eta(t),t\in[0;1]$ be an integrator. Then there exist a Wiener process $w(t),t\in[0;1]$ and a continuous linear operator $A$ in the space $L_{2}([0;1])$ such that
\begin{equation*}
\begin{aligned}
\eta(t)=\int_{0}^{1}A(\mathbf{1}_{[0;t]})(s)dw(s).
\end{aligned}
\end{equation*}
\end{lemma}
Consider a case when $\|A\|<1$. Denote
\begin{equation*}
\begin{aligned}
&Q=(I-A^{*}A)^{\frac{1}{2}},\\
&\zeta(T)=\int_{0}^{1}Q(\mathbf{1}_{[0;1]})(s)dw'(s),
\end{aligned}
\end{equation*}
where $w'$ is a Wiener process independent from $w$. It can be easily checked that, the sum
\begin{equation*}
\begin{aligned}
\widetilde{w}(t)=\eta(t)+\zeta(t), \quad t\in[0;1]
\end{aligned}
\end{equation*}
is a Wiener process. Suppose that $\Phi:C([0;1])\rightarrow \R$ is a bounded measurable function. Then, formally
\begin{equation*}
\begin{aligned}
\E\Phi(\eta)=\E(\Phi(\widetilde{w})|\zeta=0).
\end{aligned}
\end{equation*}
Note that, the conditional expectation with respect to $\zeta$ can be calculated as a second quantization operator for functionals from $\widetilde{w}$. Namely, since $\zeta$ was defined as
\begin{equation*}
\begin{aligned}
\zeta(t)=\int_{0}^{1}Q(\mathbf{1}_{[0;1]})(s)dw'(s),
\end{aligned}
\end{equation*}
then for arbitrary $h\in L_{2}([0;1])$
\begin{equation*}
\begin{aligned}
\int_{0}^{1}h(s)d\widetilde{w}(s)=\int_{0}^{1}(Ah)(s)dw(s)+\int_{0}^{1}(Qh)(s)dw'(s),
\end{aligned}
\end{equation*}
and
\begin{equation*}
\begin{aligned}
\int_{0}^{1}(Qh)(s)dw'(s)=\int_{0}^{1}h(s)d\zeta(s),
\end{aligned}
\end{equation*}
is measurable with respect to $\zeta$. Consequently, the conditional expectation of the stochastic exponent
\begin{equation*}
\begin{aligned}
&\E(\exp\{\int_{0}^{1}h(s)d\widetilde{w}(s)-\frac{1}{2}\int_{0}^{1}h^{2}(s)ds\}|\zeta)\\
&=\exp\{\int_{0}^{1}(Qh)(s)dw'(s)-\frac{1}{2}\int_{0}^{1}(Qh)^{2}(s)ds\}\\
&=\Gamma(Q)\exp\{\int_{0}^{1}h(s)dw'(s)-\frac{1}{2}\int_{0}^{1}h^{2}(s)ds\}.
\end{aligned}
\end{equation*}
Here $\Gamma(Q)$ is an operator of second quantization acting in the space of square integrable functionals from $w'$ \cite{Sim74}. Since the previous equality holds for all stochastic exponents, it remains to be ture for all square integrable functionals from $w'$ . Let $\alpha$ be such functional and has an It$\hat{o}$-Wiener expansion
\begin{equation*}
\begin{aligned}
\alpha= \E \alpha + \sum_{n=1}^{\infty}\int_{\Delta_{n}}a_{n}(t_{1},\cdots,t_{n})dw'(t_{1})\cdots dw'(t_{n}).
\end{aligned}
\end{equation*}
Then
\begin{equation*}
\begin{aligned}
\Gamma(Q)\alpha= \E \alpha + \sum_{n=1}^{\infty}\int_{\Delta_{n}}Q^{\otimes n}a_{n}(t_{1},\cdots,t_{n})dw'(t_{1})\cdots dw'(t_{n}).
\end{aligned}
\end{equation*}
Calculation $\Gamma(Q)\alpha$ when $\zeta=0$ is the same as calculation when $w'=0$. Suppose that the kernel $b_{n}\in L_{2}([0;1]^{n})$ is symmetric. To find the value of (when it is well-defined)
\begin{equation*}
\begin{aligned}
\int_{\Delta_{n}}b_{n}(t_{1},\cdots,t_{n})dw'(t_{1})\cdots dw'(t_{n}),
\end{aligned}
\end{equation*}
let us prosed as follows. Let $\{e_{k};k\geq1\}$ be am orthonormal basis in $L_{2}([0;1])$. Denote by
\begin{equation*}
\begin{aligned}
\xi_{k}=\int_{0}^{1}e_{k}(s)dw'(s), \quad k\geq1
\end{aligned}
\end{equation*}
\begin{equation*}
\begin{aligned}
b_{k_{1},\cdots,k_{n}}=\int_{0}^{1}\stackrel{n}{\cdots}\int_{0}^{1}b_{n}(t_{1},\cdots,t_{n})e_{k_{1}}(t_{1})\cdots e_{k_{n}}(t_{n})dt_{1}\cdots dt_{n}, \quad k_{1},\cdots,k_{n}\geq1.
\end{aligned}
\end{equation*}
Then
\begin{equation*}
\begin{aligned}
&\int_{\Delta_{n}}b_{n}(t_{1},\cdots,t_{n})dw'(t_{1})\cdots dw'(t_{n})\\
&=\frac{1}{n!}\int_{0}^{1}\stackrel{n}{\cdots}\int_{0}^{1}b_{n}(t_{1},\cdots,t_{n})dw'(t_{1})\cdots dw'(t_{n})\\
&=\frac{1}{n!}\sum_{k_{1},\cdots,k_{n}=1}^{\infty}b_{k_{1},\cdots,k_{n}} \xi_{k_{1}}*\cdots*\xi_{k_{n}}.
\end{aligned}
\end{equation*}
Here the multiple integral
\begin{equation*}
\begin{aligned}
\int_{0}^{1}\stackrel{n}{\cdots}\int_{0}^{1}b_{n}(t_{1},\cdots,t_{n})dw'(t_{1})\cdots dw'(t_{n})
\end{aligned}
\end{equation*}
is defined in Skorokhod sense and
\begin{equation*}
\begin{aligned}
\xi_{k_{1}}*\cdots*\xi_{k_{n}},
\end{aligned}
\end{equation*}
is a Wick product of $\xi_{k_{1}},\cdots,\xi_{k_{n}}$. Such product is obtained from the usual $\xi_{k_{1}},\cdots,\xi_{k_{n}}$ by collecting the same multiplies and substituiting the power by Hermite polynomial of the same degree. For example,
\begin{equation*}
\begin{aligned}
\xi_{1}*\xi_{2}*\xi_{3}=H_{2}(\xi_{1})\cdot\xi_{2}=(\xi^{2}_{1}-1)\xi_{2}.
\end{aligned}
\end{equation*}
Hence the value
\begin{equation*}
\begin{aligned}
&\int_{0}^{1}\stackrel{n}{\cdots}\int_{0}^{1}b_{n}(t_{1},\cdots,t_{n})dw'(t_{1})\cdots dw'(t_{n})|_{w'=0}\\
&=\sum_{k_{1},\cdots,k_{n}=1}^{\infty}b_{k_{1},\cdots,k_{n}} \xi_{k_{1}}*\cdots*\xi_{k_{n}}|_{\xi_{1}=\xi_{2}=\cdots=0}.
\end{aligned}
\end{equation*}
Using the generating function for Hermite polynomials
\begin{equation*}
\begin{aligned}
\exp(xy-\frac{1}{2}x^{2})=\sum_{n=0}^{\infty}\frac{x^{n}}{n!}H_{n}(y),
\end{aligned}
\end{equation*}
one can easily check that for $n\geq0$
\begin{equation*}
\begin{aligned}
H_{2n+1}(0)=0, \quad H_{2n}(0)=(-1)^{n}(2n-1)!!.
\end{aligned}
\end{equation*}
Consequently,
\begin{equation*}
\begin{aligned}
&\xi_{k_{1}}*\cdots*\xi_{k_{n}}|_{\xi_{1}=\xi_{2}=\cdots=0}\\
&=\E(i\xi_{k_{1}})\cdots(i\xi_{k_{n}}),
\end{aligned}
\end{equation*}
where $i^{2}=-1$. Hence (if the corresponding series converges and it is allowed to change expectation and infinite sum)
\begin{equation*}
\begin{aligned}
&\int_{0}^{1}\stackrel{n}{\cdots}\int_{0}^{1}b_{n}(t_{1},\cdots,t_{n})dw'(t_{1})\cdots dw'(t_{n})|_{w'=0}\\
&=\E\sum_{k_{1},\cdots,k_{n}=1}^{\infty}b_{k_{1},\cdots,k_{n}}(i\xi_{k_{1}})\cdots(i\xi_{k_{n}})\\
&=i^{n} \E\int_{0}^{1}\stackrel{n}{\cdots}\int_{0}^{1}b_{n}(t_{1},\cdots,t_{n})dw'(t_{1})\circ dw'(t_{2})\circ \cdots \circ dw'(t_{n}).
\end{aligned}
\end{equation*}
Where in the right hand side the Ogawa symmetric integral is used. This integral in simple cases coinsides with the Stratonowich integral. So now our strategy in the calculation of probabilities related to the integrator $\eta$ is the following. Firstly, present the probability as
\begin{equation*}
\begin{aligned}
\E\Phi(\eta).
\end{aligned}
\end{equation*}
Then find operator
\begin{equation*}
\begin{aligned}
Q=\sqrt{I-A^{*}A}
\end{aligned}
\end{equation*}
and
\begin{equation*}
\begin{aligned}
\Gamma(Q)\Phi(w),
\end{aligned}
\end{equation*}
as an It$\hat{o}$-Wiener series
\begin{equation*}
\begin{aligned}
\Gamma(Q)\Phi(w)=\sum_{0}^{\infty}\int_{\Delta_{n}}b_{n}(t_{1},\cdots,t_{n})dw(t_{1}) \cdots dw(t_{n}).
\end{aligned}
\end{equation*}
Finally,
\begin{equation*}
\begin{aligned}
\E\Phi(\eta)=\E \sum_{0}^{\infty}i^{n}\int_{\Delta_{n}}b_{n}(t_{1},\cdots,t_{n})dw(t_{1})\circ dw(t_{2})\circ \cdots \circ dw(t_{n}).
\end{aligned}
\end{equation*}
Correspondingly to such strategy the paper is divided into two sections. In the first section we consider the functionals from $w$ and its second quantization, in the second section the concrete example and estimations for expectation of symmetric integrals series are given.

\section{Second quantization for functionals from Brownian motion in the tube}

In this section we calculate the expectation of functionals from integrators. Let us suppose that the function
\begin{equation*}
\begin{aligned}
\Phi:C([0;1])\rightarrow\R,
\end{aligned}
\end{equation*}
is bounded and measurable. Consider an integrator $x$ which is generated by the operator $A$ with the norm not greater than one. Define operator $R=\sqrt{I-A^{*}A}$ as before and construct the integrator $y$ based on $R$. Then, formally,
\begin{equation*}
\begin{aligned}
\E\Phi(x)=\E(\Phi(x+y(w_{1}))|y(w_{1})=0)=\Gamma(R)\Phi(\widetilde{w})_{\widetilde{w}=0}.
\end{aligned}
\end{equation*}
Of cause to make this equalities rigorous we need some good properties of the conditional expectation
\begin{equation*}
\begin{aligned}
\E(\Phi(x+y(w_{1}))|y(w_{1})),
\end{aligned}
\end{equation*}
and the random variable
\begin{equation*}
\begin{aligned}
\Gamma(R)\Phi(\widetilde{w}),
\end{aligned}
\end{equation*}
as the functions from $y$. In the next sections we will consider the cases where they have continuous versions and, consequently, make the formal equality under $y=0$ rigorous. But now one can check the simple example, when instead of Wiener processes we use standard Gaussian variables.

\begin{example}
Consider standard Gaussian variable $\xi$. The It$\hat{o}$-Wiener expansion of square-integrable functionals from $\xi$ now is expansion via the orthogonal system of Hermite polynomials
\begin{equation*}
\begin{aligned}
\alpha=\sum_{n=0}^{\infty}a_{n}H_{n}(\xi).
\end{aligned}
\end{equation*}
Here
\begin{equation*}
\begin{aligned}
H_{n}(u)=(-1)^{n}e^{\frac{u^{2}}{2}}(\frac{d}{du})^{n}e^{-\frac{u^{2}}{2}}, \quad n\geq0.
\end{aligned}
\end{equation*}
Since $\xi$ is a random variable, then instead of $L_{2}([0;1])$ the one-dimensional space $\R$ must be taken. The operators in this space are only multiplications on the constant. It can be checked that for $|c|\leq1$
\begin{equation*}
\begin{aligned}
\Gamma(c)\alpha=\sum_{n=0}^{\infty}c^{n}a_{n}H_{n}(\xi).
\end{aligned}
\end{equation*}
Now the analog of integrator related to the operator of multiplication on $c$ is simply $c\xi$. Let us find the expectation of $H_{n}(c\xi)$ using the proposed method. Consider a standard Gaussian variable $\xi'$ which is independent from $\xi$, put $c'=\sqrt{1-c^{2}}$. Then
\begin{equation*}
\begin{aligned}
\E(H_{n}(c\xi+c'\xi')|\xi')=(c')^{n}H_{n}(\xi'),
\end{aligned}
\end{equation*}
where both sides are continuous functions from $\xi'$. Hence
\begin{equation*}
\begin{aligned}
\E(H_{n}(c\xi)|\xi')=(c')^{n}H_{n}(0).
\end{aligned}
\end{equation*}
For $n=2k+1$
\begin{equation*}
\begin{aligned}
H_{n}(0)=0
\end{aligned}
\end{equation*}
and for $n=2k$
\begin{equation*}
\begin{aligned}
\E(H_{2k}(c\xi))=(2k-1)!!(1-c^{2})^{k}.
\end{aligned}
\end{equation*}
\end{example}

The same we suppose to do in general case. Of cause existence of continuous version of the studied functionals is the problem, which will be solved in some partial cases. So, the main aim of the paper is to find the probability when multi-dimensional integrator reach the boundary of the domain $\mathcal{D}$ before the time horizon one. To do this using above mentioned approach we need to check continuity of the corresponding conditional probability and second quantization. The last part establishes the existence of continuous versions and the main result.

\section{It$\hat{o}$-Wiener expansion of the hitting event}

As it was described in the introduction we are looking for the It$\hat{o}$-Wiener expansion of the following functional from the standard  Wiener process $w$ in the domain $\mathcal{D}\subset\R^{d}$ which has a smooth boundary ($\mathcal{D}$ can be unbounded, for example, $\mathcal{D}$ can be a complement to a ball). For $x\in \mathcal{D}$ and $s\in[0;1]$ put
\begin{equation*}
\begin{aligned}
W_{x,s}=1,
\end{aligned}
\end{equation*}
if there exists such $\tau\in[s;1]$ that
\begin{equation*}
\begin{aligned}
x+w(\tau)-w(s)\in\partial \mathcal{D}.
\end{aligned}
\end{equation*}

The expectation $\E W_{x,s}$ is a probability that Wiener process starting from $x$ will hit boundary $\partial \mathcal{D}$ up to the time $1$. Since our aim is to find such probability for integrator we will find firstly the It$\hat{o}$-Wiener expansion for $W_{x,s}$ with the Fourier-Wiener transform. Recall that now it is defined with the help of stochastic exponent
\begin{equation*}
\begin{aligned}
\varepsilon(h)=\exp\{\int_{0}^{1}(h(t),dw(t))-\frac{1}{2}\int_{0}^{1}\|h(t)\|^{2}dt\},
\end{aligned}
\end{equation*}
which is defined for $h\in L_{2}([0;1],\R^{d})$, $(\cdot,\cdot)$ and $\|\cdot\|$ are the scalar product and norm in $\R^{d}$. The Fourier-Wiener transform of the square integrable functional $\alpha$ from $w$ is defined as
\begin{equation*}
\begin{aligned}
\mathcal{T}(\alpha)(h)=\E\alpha\varepsilon(h).
\end{aligned}
\end{equation*}
 It is related to It$\hat{o}$-Wiener expansion of $\alpha$ in the following way
\begin{equation*}
\begin{aligned}
\mathcal{T}(\alpha)(h)=\sum_{n=0}^{\infty}\int_{0}^{1}\stackrel{n}{\cdots}\int_{0}^{1}(a_{n}(t_{1},\cdots,t_{n}),h(t_{1})\otimes\cdots\otimes h(t_{n}))_{n}dt_{1}\cdots dt_{n}.
\end{aligned}
\end{equation*}
Here $\otimes$ denote tensor product in $\R^{d}$ and $(\cdot,\cdot)_{n}$ is a scalar product in $(\R^{d})^{\otimes n}$. Note that due to equality
\begin{equation*}
\begin{aligned}
\E\alpha^{2}=\sum_{n=0}^{\infty}\int_{0}^{1}\stackrel{n}{\cdots}\int_{0}^{1}\|a_{n}(t_{1},\cdots,t_{n})\|^{2}_{n}dt_{1}\cdots dt_{n},
\end{aligned}
\end{equation*}
where $\|\cdot\|_{n}$ is the norm in $(\R^{d})^{\otimes n}$, the Fourier-Wiener transform $\mathcal{T}(\alpha)$ is an analitic function on $H$. Consequently, to find It$\hat{o}$-Wiener expansion of $\alpha$ it is enough to find Taylor expansion of $\mathcal{T}(\alpha)$. In our case we will find $\mathcal{T}(W_{x,s})(h)$ for $h\in C^{1}([0;1],\R^{d})$. The following statement holds.

\begin{lemma}\label{lemma3.1}
For fixed $h\in C^{1}([0;1],\R^{d})$ the function
\begin{equation*}
\begin{aligned}
U^{h}_{x,t}=\mathcal{T}(W_{x,t})(h),
\end{aligned}
\end{equation*}
satisfies boundary value problem
\begin{equation*}
\begin{aligned}
    \left\{
      \begin{array}{lcl}
         \frac{\partial}{\partial{t}}U^{h}_{x,t}=-\frac{1}{2}\triangle_{x}U^{h}_{x,t}-(h(t),\nabla_{x}U_{x,t}), \\
         U^{h}_{x,1}=0, \quad x\in \mathcal{D},\\
         U_{x,t}=1, \quad x\in \partial \mathcal{D}.\\
      \end{array}
    \right.
\end{aligned}
\end{equation*}
This means that $U$ is one time continuously differentiable with respect to $t$ and two times with respect to $x$ inside $\mathcal{D}\times[0;1]$ and continuous on $\mathcal{D}\times[0;1]$ and satisfying equation inside $\mathcal{D}\times[0;1]$.
\end{lemma}

\begin{proof}
To check the statement of the lemma it is enough to recall that the stochastic exponent $\varepsilon(h)$ is a probability density and
\begin{equation*}
\begin{aligned}
\mathcal{T}(W_{x,t})(h)=\E W_{x,t}\varepsilon(h)=\E W^{h}_{x,t},
\end{aligned}
\end{equation*}
where $W^{h}_{x,t}$ is constructed from the process
\begin{equation*}
\begin{aligned}
w(t)+\int_{0}^{t}h(s)ds, \quad t\in[0;1].
\end{aligned}
\end{equation*}
The expectation
\begin{equation*}
\begin{aligned}
\E W^{h}_{x,t}
\end{aligned}
\end{equation*}
is equal to the probability that there exists such $\tau\in[t;1]$ for which
\begin{equation*}
\begin{aligned}
x+w(\tau)-w(t)+\int_{t}^{\tau}h(s)ds\in\partial \mathcal{D}.
\end{aligned}
\end{equation*}
Now the statement of the lemma becomes to be known \cite{Ric94}. Lemma is proved.
\end{proof}

The statement of the Lemma \ref{lemma3.1} allows to find $\mathcal{T}(W_{x,t})(h)$ using Duhamel series. Let $G(s,t,x,y), 0\leq s\leq t \leq 1, x,y\in\overline{\mathcal{D}}$ is a Green function \cite{Ric94} for heat equation in $\mathcal{D}$. Denote by $G'(s,t,x,y)$ the normal interior derivative of $G$ at the point $y\in\partial \mathcal{D}$. Then
\begin{equation*}
\begin{aligned}
U^{0}_{x,t}=\int_{t}^{1}\int_{\partial \mathcal{D}}G'(t,r,x,y)\delta(dy)dr.
\end{aligned}
\end{equation*}
Now Duhamel series for $U^{h}_{x,t}$ looks like
\begin{equation*}
\begin{aligned}
U^{h}_{x,t}=&U^{0}_{x,t}+\sum_{n=1}^{\infty}\int_{t\leq r_{1}\leq\cdots\leq r_{n}\leq1}\int_{\partial \mathcal{D}}\stackrel{n}{\cdots}\int_{\partial \mathcal{D}}G(t,r_{1},x,z_{1})\\
&(\nabla_{z_{1}}G(r_{1},r_{2},z_{1},z_{2})\otimes\cdots\otimes\nabla_{z_{n}}U^{0}_{z_{n},r_{n}},h(r_{1})\otimes\cdots\otimes h(r_{n}))_{n}dz_{1}\cdots dz_{n}dr_{1}\cdots dr_{n}.
\end{aligned}
\end{equation*}
As it was mentioned above the summands in obtained series gives us the kernels from the It$\hat{o}$-Wiener expansion of $W_{x,t}$.

\section{Existence of continuous version}

In this section we establish the existence of continuous version of the functionals related to conditional expectation of $W_{x,t}$ and its It$\hat{o}$-Wiener expansion.

\begin{lemma}
For fixed $x\in \mathcal{D}, t\in[0;1]$ $W^{h}_{x,t}$ is continuous with respect to $h\in C([0;1], \R^{d})$ i.e. for fixed $h_{0}\in C([0;1], \R^{d})$
\begin{equation*}
\begin{aligned}
&W^{h}_{x,t}\rightarrow W^{h_{0}}_{x,t}, \quad a.s.,
&\|h-h_{0}\|_{\infty}\rightarrow0.
\end{aligned}
\end{equation*}
Here
\begin{equation*}
\begin{aligned}
\|h\|_{\infty}=\max_{[0;1]}\|h(s)\|.
\end{aligned}
\end{equation*}
\end{lemma}
\begin{proof}
Let $x\in \mathcal{D}, t\in[0;1]$ and $h\in C([0;1], \R^{d})$ be fixed. Define the stopping moment $\tau_{h}$ as a first moment from the interval $[t;1)$ when the process
\begin{equation*}
\begin{aligned}
x+w(r)-w(t)+\int_{t}^{r}h(s)ds, \quad r\in[t;1],
\end{aligned}
\end{equation*}
firstly hits $\partial \mathcal{D}$ or $1$ if it does not hit $\partial \mathcal{D}$ or hits it exactly at the time $1$, where
\begin{equation*}
\begin{aligned}
\tau_{h}=\inf\{r\geq t: x+w(r)-w(t)+\int_{t}^{r}h(s)ds\in \partial \mathcal{D} \wedge 1\}.
\end{aligned}
\end{equation*}
Fix $h_{0}\in C([0;1], \R^{d})$ and consider several case. Firstly, suppose that $\tau_{h_{0}}<1$. Note, that due to the strong Markov property the process
\begin{equation*}
\begin{aligned}
w(\tau+s)-w(\tau), \quad s\geq0,
\end{aligned}
\end{equation*}
is again a standard Brownian motion in $\R^{d}$. Consequently, for all $u\in\R^{d}, u\neq0$ the scalar process
\begin{equation*}
\begin{aligned}
(w(\tau+s)-w(\tau),u), \quad s\geq0,
\end{aligned}
\end{equation*}
satisfies the law of iterated logarithm. Then, since $\partial \mathcal{D}$ is smooth and the integral from $h_{0}$ also is a smooth function, with probability one the process
\begin{equation*}
\begin{aligned}
x+w(r)-w(t)+\int_{t}^{r}h(s)ds, \quad r\in[t;1],
\end{aligned}
\end{equation*}
visits $\mathcal{D}$ and $\R^{d}\backslash\overline{\mathcal{D}}$ at any small interval of the time after $\tau_{h_{0}}$. Consequently, for such $\omega$ that $\tau_{h_{0}}(\omega)<1$ and the law of iterated logarithm holds with the vector $u$ chosen as a unit normal to $\partial \mathcal{D}$ at the point
\begin{equation*}
\begin{aligned}
x+w(\tau_{h_{0}})-w(t)+\int_{\tau_{h_{0}}}^{t}h_{0}(s)ds,
\end{aligned}
\end{equation*}
there exists $\delta=\delta(\omega)>0$ such that for such $h\in C([0;1], \R^{d})$ for which
\begin{equation*}
\begin{aligned}
&\|h-h_{0}\|_{\infty}<\delta(\omega),\\
&W^{h}_{x,t}=1.
\end{aligned}
\end{equation*}
Note that
\begin{equation*}
\begin{aligned}
P\{x+w(1)-w(t)+\int_{t}^{1}h_{0}(s)ds\in\partial \mathcal{D}\}=0.
\end{aligned}
\end{equation*}
Consider last case when $r\in[t;1]$:
\begin{equation*}
\begin{aligned}
x+w(r)-w(t)+\int_{t}^{r}h_{0}(s)ds\notin\partial \mathcal{D}.
\end{aligned}
\end{equation*}
For such $\omega$ there exists $\delta=\delta(\omega)>0$ for which $\forall r\in[t;1]$:
\begin{equation*}
\begin{aligned}
\rho(x+w(r)-w(t)+\int_{t}^{r}h_{0}(s)ds, \partial D)>\delta,
\end{aligned}
\end{equation*}
where $\rho$ means the distance. Hence, for such $\omega$ there exists $\delta_{1}=\delta_{1}(\omega)>0$ that for $h\in C([0;1], \R^{d})$
\begin{equation*}
\begin{aligned}
&\|h-h_{0}\|_{\infty}<\delta_{1},\\
&W^{h}_{x,t}=1.
\end{aligned}
\end{equation*}
Lemma is proved.
\end{proof}

The same statement holds for certain integrators. Consider the linear operator $A$ in $L_{2}([0;1], \R^{d})$ of the following form
\begin{equation*}
\begin{aligned}
(Ah)(t)=h(t)+\int_{0}^{1}q(t,s)h(s)ds,
\end{aligned}
\end{equation*}
where $q$ is a continuous kernel on $[0;1]^{2}$ satisfying inequality
\begin{equation*}
\begin{aligned}
\max_{[0;1]^{2}}|q(t,s)|<1.
\end{aligned}
\end{equation*}
Define integrator
\begin{equation*}
\begin{aligned}
\eta(t)=\int_{0}^{1}A(\mathbf{1}_{[0;t]})(r)dw(r).
\end{aligned}
\end{equation*}

\begin{lemma}
For fixed $x\in \mathcal{D}, t\in[0;1], h_{0}\in C([0;1])$, the random variables $W^{A,h}_{x,t}$ constructed in the same way as $W^{h}_{x,t}$ converge a.s. to $W^{A,h_{0}}_{x,t}$ when $\|h-h_{0}\|\rightarrow0$.
\end{lemma}

\begin{proof}
The proof follows from the statement of the previous lemma and the fact that now distribution of $w$ and $\eta$ in $C([0;1],\R^{d})$ are equivalent \cite{Vla15}.
\end{proof}

Next lemma shows that some time It$\hat{o}$-Wiener expansion can be a continuous with respect to initial noise.

Let $R$ be a finite dimensional bounded linear operator in $L_{2}([0;1])$. Denote by the same letter $R$ the operator in $L_{2}([0;1],\R^{d})$ which is acting as follows
\begin{equation*}
\begin{aligned}
Rh=(Rh^{1},Rh^{2},\cdots,Rh^{d}),
\end{aligned}
\end{equation*}
where
\begin{equation*}
\begin{aligned}
h(t)=(h^{1}(t),\cdots,h^{d}(t)).
\end{aligned}
\end{equation*}
Let $\alpha$ be a square-integrable functional from $w$.

\begin{lemma}\label{lemma4.3}
Suppose that the operator norm $\|R\|<1$, and
\begin{equation*}
\begin{aligned}
R=\sum_{k=1}^{m}\beta_{k}e_{k}\otimes e_{k}, \quad \{e_{k}\}\subset C^{1}([0;1]).
\end{aligned}
\end{equation*}
Then the It$\hat{o}$-Wiener series for $\Gamma(R)\alpha$ is a continuous function from $\Gamma(R)w$.
\end{lemma}

\begin{proof}
Now $\Gamma(R)w$ is a Gaussian process which depends on a finite number of standard independent Gaussian variables of the kind
\begin{equation*}
\begin{aligned}
\xi^{j}_{k}=\int_{0}^{1}e_{k}(t)dw^{j}(t), \quad k=1,\cdots,m, \quad j=1,\cdots,d,
\end{aligned}
\end{equation*}
where the orthonormal system $\{e_{k};k=1,\cdots,m\}$ used. So, due to the condition $\|R\|<1$ the random variable $\Gamma(R)\alpha$ is infinitely many times stochastically differentiable \cite{Vla15}. Consequently, due to Sobolev embedding theorem \cite{Vla15}, $\Gamma(R)\alpha$ can be represented as a continuous functional
\begin{equation*}
\begin{aligned}
\Gamma(R)\alpha=F(\xi^{j}_{k}, \quad k=1,\cdots,m, \quad j=1,\cdots,d),\quad F\in C(\R^{m\times d}).
\end{aligned}
\end{equation*}
Lemma is proved.
\end{proof}

Now we can make rigorous the arguments presented in the introduction. Let $R$ be a finite-dimensional operator in $L_{2}([0;1])$ with the norm $\|R\|<1$. Take the continuous linear operator $A$ such that for $\forall h\in L_{2}([0;1])$,
\begin{equation*}
\begin{aligned}
\|Ah\|^{2}+\|Rh\|^{2}=1.
\end{aligned}
\end{equation*}
Consider integrator $\eta$ constructed with the help of $A$. The following theorem holds.

\begin{theorem}
For arbitrary $x\in \mathcal{D}$,
\begin{equation*}
\begin{aligned}
P\{\exists\tau\in[0;1]:x+\eta(\tau)\in\partial \mathcal{D}\}=\Gamma(R)W_{x,0}|_{w=0}.
\end{aligned}
\end{equation*}
\end{theorem}

\begin{proof}
Consider, as in introduction, new standard d-dimensional process $w'$ independent from $w$ and construct new integrator
\begin{equation*}
\begin{aligned}
\zeta(t)=\int_{0}^{1}R(\mathbf{1}_{[0;t]})(s)dw'(s).
\end{aligned}
\end{equation*}
Then $\zeta$ has smooth trajectories. So, due to the lemma \ref{lemma2.2}, conditional expectation
\begin{equation*}
\begin{aligned}
\E (W^{\widetilde{w}}_{x,0}|\zeta),
\end{aligned}
\end{equation*}
is a continuous function of $\zeta$. From other side, by lemma \ref{lemma4.3}
\begin{equation*}
\begin{aligned}
\Gamma(R)W^{w'}_{x,0},
\end{aligned}
\end{equation*}
is a continuous function of the random variables
\begin{equation*}
\begin{aligned}
\xi^{j}_{k}=\int_{0}^{1}e_{k}dw'^{j}(t), \quad k=1,\cdots,m, \quad j=1,\cdots,d.
\end{aligned}
\end{equation*}
Hence from equality
\begin{equation*}
\begin{aligned}
\E (W^{\widetilde{w}}_{x,0}|\zeta)=\Gamma(R)W^{w'}_{x,0},
\end{aligned}
\end{equation*}
one can conclude the equality under $\xi^{j}_{k}=0, \quad k=1,\cdots,m,\quad  j=1,\cdots,d$ which is the desired statement. Theorem is proved.
\end{proof}

Now use the It$\hat{o}$-Wiener expansion of $W_{x,0}$ to obtain a series representation for hitting probability of integrator. Due to representation of the operator $R$ the It$\hat{o}$-Wiener expansion of $\Gamma(R)W_{x,0}$ can be simplified. We will do it in the case $d=1$ and then give the general formula just to avoid long calculation. So, let us suppose that $d=1$ and for function $a_{n}\in L_{2}([0;1])$ consider
\begin{equation*}
\begin{aligned}
&\int_{0}^{1}\stackrel{n}{\cdots}\int_{0}^{1}(R^{\otimes n}a_{n})(t_{1},\cdots,t_{n})dw(t_{1}),\cdots,dw(t_{n})\\
&=\int_{0}^{1}\stackrel{n}{\cdots}\int_{0}^{1}\sum_{i_{1},\cdots,i_{n}=1}^{m}\beta_{i_{1}}\cdots\beta_{i_{n}}\int_{0}^{1}\stackrel{n}{\cdots}\int_{0}^{1}a_{n}(s_{1},\cdots,s_{n})\\
&\cdot e_{i_{1}}(s_{1})\cdots e_{i_{n}}(s_{n})ds_{1}\cdots ds_{n}e_{i_{1}}(t_{1})\cdots e_{i_{n}}(t_{n})dw(t_{1}),\cdots,dw(t_{n}).
\end{aligned}
\end{equation*}
For further transformations recall notation
\begin{equation*}
\begin{aligned}
\xi_{k}=\int_{0}^{1}e_{k}(t)dw(t_{1}), \quad k=1,\cdots,m.
\end{aligned}
\end{equation*}
Then the stochastic integral can be expressed in terms of Hermite polynomials from the variables $\xi_{k},k=1,\cdots,m$. Denote for the sequence of indexes $i_{1},\cdots,i_{n}$ by the $r_{1},\cdots,r_{n}$ the numbers which indicate how many times $1$ repeats among $i_{1},\cdots,i_{n}$, $2$ repeats among $i_{1},\cdots,i_{n}$ and so on. Then
\begin{equation*}
\begin{aligned}
\int_{0}^{1}\stackrel{n}{\cdots}\int_{0}^{1}e_{i_{1}}(t_{1})\cdots e_{i_{n}}(t_{n})dw(t_{1}),\cdots,dw(t_{n})=H_{r_{1}}(\xi_{1})\cdots H_{r_{m}}(\xi_{m}).
\end{aligned}
\end{equation*}
Introduce the coefficients
\begin{equation*}
\begin{aligned}
a^{n}_{r_{1}\cdots r_{m}}=\int_{0}^{1}\stackrel{n}{\cdots}\int_{0}^{1}a_{n}(t_{1}\cdots t_{n})e_{i_{1}}(t_{1})\cdots e_{i_{n}}(t_{n})dt_{1},\cdots,dt_{n}.
\end{aligned}
\end{equation*}
Finally
\begin{equation*}
\begin{aligned}
&\int_{0}^{1}\stackrel{n}{\cdots}\int_{0}^{1}(R^{\otimes n}a_{n})(t_{1},\cdots,t_{n})dw(t_{1}),\cdots,dw(t_{n})\\
&=\sum_{r_{1}+\cdots+r_{m}=n}\frac{n!}{r_{1}!\cdots r_{m}!}\beta_{1}^{r_{1}}\cdots\beta_{m}^{r_{m}}a^{n}_{r_{1}\cdots r_{m}}H_{r_{1}}(\xi_{1})\cdots H_{r_{m}}(\xi_{m}).
\end{aligned}
\end{equation*}
Let us calculate the second moment of the stochastic integral using the obtained representation.
\begin{equation*}
\begin{aligned}
&\E (\int_{0}^{1}\stackrel{n}{\cdots}\int_{0}^{1}(R^{\otimes n}a_{n})(t_{1},\cdots,t_{n})dw(t_{1}),\cdots,dw(t_{n}))^{2}\\
&=\sum_{r_{1}+\cdots+r_{m}=n}\frac{n!^{2}}{r_{1}!\cdots r_{m}!}\beta_{1}^{2r_{1}}\cdots\beta_{m}^{2r_{m}}(a^{n}_{r_{1}\cdots r_{m}})^{2}.
\end{aligned}
\end{equation*}
Note, that
\begin{equation*}
\begin{aligned}
\E \alpha^{2}\geq\sum_{n=0}^{\infty}\sum_{r_{1}+\cdots+r_{m}=n}\frac{n!^{2}}{r_{1}!\cdots r_{m}!}(a^{n}_{r_{1}\cdots r_{m}})^{2}.
\end{aligned}
\end{equation*}
Recall the estimation for Hermite polynomials \cite{Gab98}
\begin{equation*}
\begin{aligned}
\max_{[-1;1]}|H_{n}(x)|\leq c\cdot n^{-\frac{1}{4}} \cdot \sqrt{n!}, \quad n\geq1.
\end{aligned}
\end{equation*}
Hence for $|\xi_{k}|\leq1, k=1,\cdots,m$
\begin{equation*}
\begin{aligned}
&\sum_{r_{1}+\cdots+r_{m}=n}\frac{n!}{r_{1}!\cdots r_{m}!}|\beta_{1}^{r_{1}}\cdots\beta_{m}^{r_{m}}||a^{n}_{r_{1}\cdots r_{m}}||H_{r_{1}}(\xi_{1})\cdots H_{r_{m}}(\xi_{m})|\\
&\leq c ^{m}\sum_{r_{1}+\cdots+r_{m}=n}\frac{n!}{\sqrt{r_{1}!\cdots r_{m}!}}|a^{n}_{r_{1}\cdots r_{m}}||\beta_{1}|^{r_{1}}\cdots|\beta_{m}|^{r_{m}}\\
&\leq c ^{m}(\E \alpha^{2})^{\frac{1}{2}}\sum_{r_{1}+\cdots+r_{m}=n}|\beta_{1}|^{r_{1}}\cdots|\beta_{m}|^{r_{m}}.
\end{aligned}
\end{equation*}
Note that
\begin{equation*}
\begin{aligned}
&\sum_{n=0}^{\infty}\sum_{r_{1}+\cdots+r_{m}=n}|\beta_{1}|^{r_{1}}\cdots|\beta_{m}|^{r_{m}}\\
&=(\sum_{r_{1}=0}^{\infty}|\beta_{1}|^{r_{1}})\cdots(\sum_{r_{m}=0}^{\infty}|\beta_{m}|^{r_{m}})<+\infty.
\end{aligned}
\end{equation*}
Consequently It$\hat{o}$-Wiener series for $\Gamma(R)\alpha$ now converges uniformly on the set $|\xi_{k}|\leq1, k=1,\cdots,m$. Hence
\begin{equation*}
\begin{aligned}
\Gamma(R)\alpha|_{w=0}=\sum_{n=0}^{\infty}\sum_{r_{1}+\cdots+r_{m}=n}\frac{n!}{r_{1}!\cdots r_{m}!}a^{n}_{r_{1}\cdots r_{m}}H_{r_{1}}(0)\cdots H_{r_{m}}(0).
\end{aligned}
\end{equation*}

Now we can formulate the main statement of the paper. Consider operator $R$ as above. Suppose that the kernels $g^{n}_{x}\in L_{2}([0;1]^{n},\R^{d\times n})$ are obtained from the It$\hat{o}$-Wiener expansion for $W_{x,0}$. Let $\eta$ be the integrator based on the operator $A$ such that, for $\forall h\in L_{2}([0;1])$:
\begin{equation*}
\begin{aligned}
\|Ah\|^{2}+\|Rh\|^{2}=\|h\|^{2}.
\end{aligned}
\end{equation*}
Denote by
\begin{equation*}
\begin{aligned}
g^{n}_{x,R}=R^{\otimes d\times n}g^{n}_{x}.
\end{aligned}
\end{equation*}
Then $g^{n}_{x,R}$ is a kernel from $L_{2}([0;1], \R^{d})$. More over,
\begin{equation*}
\begin{aligned}
&\int_{0}^{1}\stackrel{n}{\cdots}\int_{0}^{1}(g^{n}_{x,R}(t_{1},\cdots,t_{n}),dw(t_{1})\otimes\cdots\otimes dw(t_{n}))\\
&=\sum_{k_{1}+\cdots+k_{d}=n}\sum_{j^{1}_{1}+\cdots+j^{1}_{m}=k_{1}}\cdots\sum_{j^{d}_{1}+\cdots+j^{d}_{m}=k_{d}}\\
&\frac{n!}{j^{1}_{1}!\cdots j^{1}_{m}!\cdots j^{d}_{1}!\cdots j^{d}_{m}!}\alpha_{j^{1}_{1}\cdots j^{1}_{m}\cdots j^{d}_{1}\cdots j^{d}_{m}}H_{j^{1}_{1}}(\xi_{1}^{1})\cdots H_{j^{1}_{m}}(\xi_{1}^{m})\cdots H_{j^{d}_{1}}(\xi_{d}^{1})\cdots H_{j^{d}_{m}}(\xi_{d}^{m}),
\end{aligned}
\end{equation*}
where $\{\alpha_{j^{1}_{1}\cdots j^{1}_{m}\cdots j^{d}_{1}\cdots j^{d}_{m}}\}$ are the coordinates of $g^{n}_{x,R}$ in the basis $\{e_{1},\cdots,e_{m}\}$ and, as before
\begin{equation*}
\begin{aligned}
\xi^{j}_{k}=\int_{0}^{1}e_{k}dw^{j}(t), \quad k=1,\cdots,m, \quad j=1,\cdots,d.
\end{aligned}
\end{equation*}
Finally, the following statement holds.

\begin{theorem}
\begin{equation*}
\begin{aligned}
&P\{\exists\tau\in[0;1]:x+\eta(\tau)\in\partial \mathcal{D}\}\\
&=\sum_{n=0}^{\infty}\sum_{k_{1}+\cdots+k_{d}=n}\sum_{j^{1}_{1}+\cdots+j^{1}_{m}=k_{1}}\cdots\sum_{j^{d}_{1}+\cdots+j^{d}_{m}=k_{d}}\\
&\frac{n!}{j^{1}_{1}!\cdots j^{1}_{m}!\cdots j^{d}_{1}!\cdots j^{d}_{m}!}\alpha_{j^{1}_{1}\cdots j^{1}_{m}\cdots j^{d}_{1}\cdots j^{d}_{m}}H_{j^{1}_{1}}(0)\cdots H_{j^{1}_{m}}(0)\cdots H_{j^{d}_{1}}(0)\cdots H_{j^{d}_{m}}(0).
\end{aligned}
\end{equation*}
\end{theorem}

\section*{Declarations}

\noindent \textbf{Ethics approval}

\noindent Not applicable.

\noindent \textbf{Competing interests}

\noindent The author declares that there is no conflict of interest or competing interest.

\noindent \textbf{Authors' contributions}

\noindent All authors contributed equally to this work.

\noindent \textbf{Funding}

\noindent There is no funding source for this manuscript.

\noindent \textbf{Availability of data and materials}

\noindent Data sharing is not applicable to this article as no data sets were generated or analyzed during the current study.

\end{document}